\documentclass[a4paper,11pt,pdf]{amsart}
\usepackage{enumerate, amsmath, amsfonts, amssymb, amsthm, thmtools, wasysym, graphics, graphicx, xcolor, frcursive,comment,bbm}

\usepackage{etex}
\usepackage{lscape}
\usepackage{tikz-cd}

\makeatletter
\newcommand{\thickhline}{%
	\noalign {\ifnum 0=`}\fi \hrule height 1pt
	\futurelet \reserved@a \@xhline
}

\definecolor{darkblue}{rgb}{0.0,0,0.7} 

\definecolor{darkred}{rgb}{0.7,0,0} 
\usepackage{hyperref}
\usepackage[all]{xy}
\usepackage[T1]{fontenc}

\usepackage{vmargin}

\usepackage{caption,lipsum}
\usepackage{graphicx}                  
\usepackage{pstricks,pst-plot,pst-text,pst-tree,pst-eps,pst-fill,pst-node,pst-math}
\usepackage{setspace}
\usepackage{multicol}

\newcommand{\darkred}{\color{darkred}} 
\newcommand{\defn}[1]{\emph{\darkred #1}}

\usepackage{etex}

\newtheorem{theorem}{Theorem}[section]
\newtheorem{prop}[theorem]{Proposition}
\newtheorem{lemma}[theorem]{Lemma}
\newtheorem{cor}[theorem]{Corollary}

\theoremstyle{definition}
\newtheorem{definition}[theorem]{Definition}
\newtheorem{rmq}[theorem]{Remark}
\newtheorem{exple}[theorem]{Example}
\newtheorem{question}[theorem]{Question}

\numberwithin{equation}{section}

\author{Thomas Gobet}
\address{Université Clermont Auvergne, LMBP, UMR 6620 (CNRS), Campus des Cézeaux, 3 place Vasarely, TSA 60026, CS 60026, 63178 Aubière cedex, France}

\title[Faithful Burau-like representations of some rank two Garside groups]{Faithful Burau-like representations of some rank two Garside groups and torus knot groups}

\begin{document}
	
\maketitle	
	
\begin{abstract}
We give a method to produce faithful representations of the groups $G(n,m)=\langle X, Y \ \vert \ X^m = Y^n \rangle$ in $\mathrm{GL}_2(\mathbb{C}[t^{\pm 1}, q^{\pm 1}])$. These groups are Garside groups and the Garside normal forms of elements of the corresponding monoid can be explicitly recovered from the matrices, in the spirit of Krammer's proof of the linearity of Artin's braid groups. We use this method to construct several explicit faithful representations of the above groups, among which a representation which generalizes the reduced Burau representation of $B_3 \cong G(2,3)$ to a large family of groups of the form $G(n,m)$ with $n$, $m$ coprime (which are torus knot groups). Like the Burau representation, this representation specializes to a representation of a reflection-like quotient that we previously introduced, called \textit{$2$-toric reflection group}. As a byproduct we get a "Burau representation" for some exceptional complex braid groups, which also shows that the  latter embed into their Hecke algebra.    
\end{abstract}	

\tableofcontents

\section{Introduction} 

The aim of this paper is to give some more evidence that torus knot groups behave like Artin groups (or more generally complex braid groups) of rank two. The most basic observation that can be done is that the $3$-strand Artin braid group $B_3$ is isomorphic to the knot group of the trefoil knot, the most elementary nontrivial (torus) knot, and that this isomorphism maps braided reflections to (powers of) meridians. More generally, Artin groups of odd dihedral type are isomorphic to torus knot groups, and the complex braid groups of the exceptional complex $2$-reflection groups $G_{12}$ and $G_{22}$ as well. Both families are distinct, but share several group-theoretic properties: for instance, it has long been known that, like Artin's braid group, torus knot groups are torsion free, with an infinite cyclic center~\cite{Schreier}, and linear. In a more modern language, it has been observed that both turn out to be \textit{Garside groups}~\cite{DP}, and form two of the most elementary, though already extremely rich, families of such groups.  

Artin groups admit a natural quotient given by the corresponding Coxeter group. In a previous work~\cite{Gobet_toric}, we showed that one can define similar quotients for torus knot groups, and that these quotients behave like (in general infinite) complex reflection groups of rank two. Such a quotient is called a~\textit{toric reflection group}, and the corresponding torus knot group behaves like its braid group. 

\medskip

In this paper, we go one step further by showing in Subsection~\ref{sub_burau_torus} that for a large family of torus knot groups, we can define an analogue of the \textit{(reduced) Burau representation} of $B_3$. 

The Burau representation~\cite{Bur} is perhaps one of the most fascinating representation of Artin's braid group. It is obtained by deforming the natural action of the corresponding Coxeter group, namely the symmetric group $S_n$, on its root system. For $n\geq 5 $ it is unfaithful~\cite{Bigelow}, for $n=4$ the faithfulness is still open, while for $n=3$ it is faithful (see~\cite[Theorem 3.15]{Bir} and the references therein). There is also a definition for all Artin groups of rank two, and it is faithful in this case~\cite{Squier, LX}. Similarly, the representation that we construct will be faithful, as expected for a family of groups that we claim to behave like Artin groups (or complex braid groups) of rank two. 

\medskip

Let $n,m \geq 2$ and let $$G(n,m)= \langle X, Y \ \vert \ X^m = Y^n \rangle.$$ When $n$ and $m$ are coprime, this is a presentation of the knot group of the torus knot $T_{n,m}$ (every torus knot, except the unknot, is isotopic to some $T_{n,m}$).  

 The following statement will be our main tool to construct faithful representations of the groups $G(n,m)$, and can be used without the assumption that $n$ and $m$ are coprime.

\begin{prop}\label{fund}
Assume that $m, n \geq 2$ are two (not necessarily coprime) integers, let $M_X, M_Y\in\mathrm{GL}_2(\mathbb{C}[t^{\pm 1}, q^{\pm 1}])$ satisfy the following three conditions: 
\begin{enumerate}
	\item We have $M_X^m = M_Y^n$, i.e., the matrices define a representation of $G(n,m)$,
	\item For all $1 \leq k \leq m-1$, the matrix $M_X^k$ has the form $$t^{nk} \begin{pmatrix} c_1(k) & c_2(k) \\ c_3(k) & c_4(k)\end{pmatrix},$$ where $c_i(k)\in \mathbb{C}$ for all $i$, and $c_2(k)\neq 0$, 
	\item For all $1 \leq k \leq n-1$, the matrix $M_Y^k$ has the form $$t^{mk} \begin{pmatrix} P_1(k) & P_2(k) \\ P_3(k) & P_4(k)\end{pmatrix},$$ where $P_i(k) \in\mathbb{C}[q^{\pm 1}]$ for all $i$, $\deg(P_i(k))\leq 0$ for $i\neq 3$, and $\deg(P_3(k))=1$. 
\end{enumerate}
Then the assignment $X \mapsto M_X$, $Y \mapsto M_Y$ defines a faithful representation of $G(n,m)$. Moreover, the representation obtained by setting $t=q$ stays faithful. 
\end{prop}

It might look difficult at first to find matrices satisfying the above conditions, but as we shall see in Section~\ref{const_rep} below, several such representations can naturally be constructed. 

The main strategy for the proof of this proposition will be to show that one can recover the Garside normal form of an element $u$ of the Garside monoid $\mathcal{M}(n,m) = \langle X, Y \ \vert \ X^m=Y^n \rangle \hookrightarrow G(n,m)$ from its matrix $M_u$ in the representation. This follows Krammer's philosophy, who used this approach to show that braid groups are linear~\cite{Krammer}--see Section~\ref{sub_phil} below for more details. Krammer's arguments are indeed purely Garside-theoretic, but it seems that his strategy has not been applied outside the particular family of Garside groups given by Artin groups of spherical type; to the best of our knowledge the above proposition gives a first example of another family of Garside groups for which one can use Krammer's approach.   

Proposition~\ref{fund} will be proven in Section~\ref{sec_proof}, after recalling a few basic properties of Garside groups in Section~\ref{sec_prelim}; in fact, the Garside-theoretic properties of $G(n,m)$ are not required for the proof since they are in fact reproven when proving Proposition~\ref{fund}, but having them in mind may be enlightening. In Section~\ref{const_rep}, we then construct various faithful representations of the groups $G(n,m)$ using Proposition~\ref{fund}, including in Subsection~\ref{sub_burau_torus} the aforementioned representation generalizing the Burau representation, which we also show to be unitarizable. The last Section~\ref{hecke} is devoted to showing that for those torus knot groups for which we can build a Burau representation, the faithfulness can be used to show that these groups embed into an analogue of Hecke algebra. In particular this will yield a faithful Burau representation for the complex braid groups of the exceptional complex reflection groups $G_{12}$ and $G_{22}$ (see Examples~\ref{g12} and~\ref{g22}).  

\medskip

\textbf{Acknowledgements} The idea to construct Burau representations for torus knot groups grew up in February 2022 while I was attending the semester program \textit{Braids} at ICERM (Providence). I thank ICERM and the organizers for the invitation and financial support. I also thank Eddy Godelle, Igor Haladjian, Abel Lacabanne, Ivan Marin, Hoel Queffelec, and Emmanuel Wagner for useful discussions.    

\section{Preliminaries}\label{sec_prelim}

In this paper, we shall be interested in representations of the following groups: 
\begin{definition}
Let $m, n \geq 2$ be two integers. We denote by $G(n,m)$ the group defined by the presentation $\langle X, Y \ \vert \ X^m = Y^n \rangle$. 	
\end{definition}

\subsection{Garside groups and their representations}\label{sub_phil} By~\cite{DP}, the groups $G(n,m)$ are Garside groups. We will not recall all the theory of Garside monoids and groups here, but only some properties that we will use later on. For the interested reader, more on Garside structures can be found in~\cite{Garside, DP}; in fact, as explained in Remark~\ref{rmq_gars} below, the fact that $G(n,m)$ is a Garside group is in fact not needed in this paper, but our approach rather reproves some important properties that $G(n,m)$ admits as a Garside group. 

Roughly speaking, saying that $G(n,m)$ is a Garside group means that $G(n,m)$ is the group of fractions of the monoid $\mathcal{M}(n,m)$ defined by the same presentation (which we will call a \textit{Garside presentation}), that the latter is cancellative and has good divisibility properties, and that there exists an element $\Delta$ (the \textit{Garside element}) whose set of left- and right-divisors coincide, and form a finite generating set $\mathrm{Div}(\Delta)$ of $\mathcal{M}(n,m)$ (and hence of $G(n,m)$), called the set of \textit{simple elements} of $\mathcal{M}(n,m)$. In our case the Garside element is given by $X^m = Y^n$, and is central (a Garside element always has a central power, and since Garside groups are torsion-free, a Garside group in particular has a non-trivial center). In this case, it is even a generator of the center, as shown by Schreier~\cite{Schreier}. 

In a Garside group $G$ attached to a Garside monoid $\mathcal{M}$, every element $g\in G$ can thus be written under the form $x^{-1}y$, where $x, y\in M$. Garside groups have a solvable word problem and every element admits canonical left- and right-normal forms, whose factors are elements of the set $\mathrm{Div}(\Delta)$. In fact, given any element $x\in \mathcal{M}(n,m)\backslash\{1\}$, the good divisibility properties of $\mathcal{M}(n,m)$ ensure that there is a unique element $x_1$ in $\mathrm{Div}(\Delta)\backslash\{1\}$, maximal for left-divisibility, that can be factored out from $x$. We thus have $x=x_1 x'$ for some (by cancellativity uniquely-defined) element $x'\in M$, and one can go on factoring $x'$. The properties of Garside monoids guarantee that this procedure terminates, yielding a a canonical decomposition $x=x_1 x_2 \cdots x_k$ with $x_i \in \mathrm{Div}(\Delta) \backslash\{1\}$, the \textit{left-Garside normal form} of $x$. Factoring out from the right yields the \textit{right-Garside normal form}. In our case we have $\mathrm{Div}(\Delta)=\{1, X, X^2, \dots, X^{m-1}, Y, Y^2, \dots, Y^{n-1}\}$.

We shall construct two-dimensional representations of $G(n,m)$, and show that they are faithful using the following procedure, which is the philosophy used by Krammer to show that Artin's $n$-strand braid groups are linear~ \cite{Krammer}: 
\begin{enumerate}
	\item To show that an action of a Garside group $G$ (with Garside monoid $\mathcal{M}$) on a set $X$ is faithful, since every element of $G$ can be written under the form $x y^{-1}$ with $x, y\in \mathcal{M}$, it suffices to show that the monoid $\mathcal{M}$ acts faithfully. 
	\item To show that the action of $\mathcal{M}$ is faithful, one can try to reconstruct the left- of -right Garside normal form of every element $m\in \mathcal{M}$ from its corresponding automorphism of $X$.  
\end{enumerate}
This approach has been successfully applied to show that several (linear or categorical) actions of braid groups or more generally spherical type Artin-Tits groups are faithful, see for instance~\cite{Digne, LX, BT, Jensen, LQ}. Since Krammer's approach is Garside-theoretic, it is natural to wonder whether it can be applied for other families of Garside groups. To the best of our knowledge, this has not been investigated before, and we shall do it below for the family of groups $G(n,m)$, which is among the easiest family of Garside groups distinct from spherical type Artin-Tits groups.   
  
\begin{rmq}\label{rmq_gars} The fact that $G(n,m)$ is a Garside group is in fact not needed \textit{a priori} in this approach, which in fact reproves these facts: the point $2$ in the above philosophy implies that $\mathcal{M}(n,m)$ embeds into a group (the group of automorphisms of $X$), hence that it is cancellative (a property that is hard to show in general for a monoid defined by generators and relations). The fact that every element of $G(n,m)$ can we written as a fraction in two elements of $\mathcal{M}(n,m)$ can also be established easily in this case using the fact that if $x, y\in \mathrm{Div}(\Delta)$, then writing $x x'= yy' = \Delta$, we have $x^{-1} y= x' {y'}^{-1}$, hence in any word in $\mathrm{Div}(\Delta)$ (recall that $X, Y\in \mathrm{Div}(\Delta)$), negative powers can be inductively moved to the left of the word. It is thus worth mentioning that the above philosophy can in fact also be used to show some properties of $G$ and $\mathcal{M}$, such as the cancellativity of $\mathcal{M}$. This was used for instance by Paris to show that Artin monoids, which in general are not Garside, embed in the corresponding groups~\cite{Paris}.     
\end{rmq}

\subsection{Properties of the groups under consideration}

As already pointed out, the groups $G(n,m)$ form a family of Garside groups that is among the easiest. Let us point out some additional features: 
\begin{itemize}
	\item If $n=2$ and $m$ is odd, then $G(n,m)$ is isomorphic to the Artin group of dihedral type $I_2(m)$.
	\item If $(n,m)=(3,4)$, respectively $(3,5)$, then $G(n,m)$ is isomorphic to the complex braid group of the exceptional complex reflection group $G_{12}$, respectively $G_{22}$ (see~\cite{Bannai}).
	\item If $n, m$ are coprime, which in particular covers all the examples above, then $G(n,m)$ is the torus knot group of the torus knot $T_{n,m}$. Since every torus knot (except the unknot) is of this form, it covers all the torus knots apart from the knot group of the unknot, which is isomorphic to $\mathbb{Z}$. 
\end{itemize}
In the case where $n$ and $m$ are coprime, since $G(2,3)$ is isomorphic to the three-strand braid group $B_3$, the groups $G(n,m)$ can be considered as generalizations of $B_3$. Note that the presentation $\langle X, Y \ \vert \ X^3 = Y^2 \rangle$ is \textit{not} the standard Artin presentation of $B_3$. The standard Artin presentation of $B_3$ can in fact be generalized when $n$ and $m$ are coprime as we now recall, and this suggests to investigate whether properties of $B_3$ can actually be generalized to the groups of the form $G(n,m)$ with $n$ and $m$ coprime. Previous work of the author~\cite{Gobet_toric} introduced and classified analogues of the complex reflection groups appearing as quotients of $B_3$ when adding torsion on the generators of the standard presentation, and studied Garside structures analogous to those that are known for $B_3$~\cite{Gobet_new}. Note that Haladjian recently generalized some of these results to a larger family of Garside groups, in such a way that they include all the complex braid groups of complex reflection groups of rank two~\cite{H1, H2}. In this paper we continue this program by taking a more representation-theoretic road: will see that for most groups $G(n,m)$ with $n$ and $m$ coprime, we can construct a representation generalizing the Burau representation of $B_3$, and show that it is faithful using the philosophy recalled in Subsection~\ref{sub_phil}. This will be done by proving Proposition~\ref{fund}--which in fact will allow us to produce faithful representations for all the groups of the form $G(n,m)$, that is, not only when $n$ and $m$ coprime (see Section~\ref{sub_faith_gen} below).

\subsection{Case where the parameters are coprime} Let $n,m \geq 2$ with $n$ and $m$ coprime. The group $G(n,m)$ admits the presentation 
\begin{equation}\label{pres_2}\langle \ x_1, x_2, \dots, x_n \ \vert\ 
	\underbrace{x_1 x_2 \cdots}_{m~\text{factors}} = \underbrace{x_2 x_3\cdots}_{m~\text{factors}} = \dots = \underbrace{x_n x_1 \cdots}_{m~\text{factors}}
	\ \rangle,\end{equation} where indices are taken modulo $n$ if $n<m$.
Since $G(n,m)\cong G(m,n)$, it also admits the presentation \begin{equation}\label{pres_3}\langle \ y_1, y_2, \dots, y_m \ \vert\ 
	\underbrace{y_1 y_2 \cdots}_{n~\text{factors}} = \underbrace{y_2 y_3\cdots}_{n~\text{factors}} = \dots = \underbrace{y_m y_1 \cdots}_{n~\text{factors}}
	\ \rangle.\end{equation}     
For $B_3 \cong G(2,3)$ (more generally for the Artin group of dihedral type $I_2(m)$ with $m$ odd, which is isomorphic to $G(2,m)$), Presentation~\ref{pres_2} is nothing but the classical Artin-Tits group presentation, while Presentation~\ref{pres_3} is nothing but its Birman-Ko-Lee or dual presentation. An explicit isomorphism between $G(n,m)$ and the group with Presentation~\ref{pres_2} is given by $X \mapsto x_1 x_2 \cdots x_n$, $Y \mapsto x_1 x_2 \cdots x_m$. We refer to~\cite[Section 3]{Gobet_new} for more details and proofs of these facts. The above two presentations are also Garside presentations. 

\begin{definition}[{\cite[Section 2.2]{Gobet_toric}}]
Let $n,m\geq 2$ be coprime and $k \geq 2$. The \defn{toric reflection group} $W(k,n,m)$ is the quotient of $G(n,m)$ obtained by adding the relations $x_i^k = 1$ in Presentation~\ref{pres_2} for all $i=1, \dots, n$. 
\end{definition} 

In fact, since $n$ and $m$ are coprime, all the $x_i$'s are conjugate in $G(n,m)$. It thus suffices to add the relation $x_1^k=1$. These groups behave like reflection groups in some sense, and the corresponding torus knot groups like their braid groups, generalizing the quotient $B_3 \twoheadrightarrow S_3$ (see~\cite{Gobet_toric}).

\begin{definition}\label{2toric} Let $n,m$ be as above. A group of the form $W(2,n,m)$ will be called a \defn{$2$-toric reflection group}.
\end{definition}

\begin{rmq}\label{rmq_mer} Let $a, b \in \mathbb{Z}$ such that $an - bm = 1$. Using the above isomorphism between $G(n,m)$ and the group with presentation~\ref{pres_2}, we see that $Y^b X^{-a}$ is mapped to $x_i^{\pm 1}$ for some $i$. Hence since all the $x_i$'s are conjugate, to get a presentation of $W(k,n,m)$ from the presentation $\langle X, Y \ \vert \ X^m = Y^n \rangle$, it suffices to add the relation $(Y^b X^{-a})^k =1$. This will be needed later on in Subsection~\ref{sub_burau_torus}. In fact, when $n$ and $m$ are coprime, the group $G(n,m)$ is isomorphic to the fundamental group of the complement of the torus knot $T_{n,m}$ in $S^3$, and $Y^b X^{-a}$ is a meridian in this group (see~\cite[Proposition~3.38(b)]{knots}).
\end{rmq}  
 
\section{Proof of Proposition~\ref{fund}}\label{sec_proof}

\subsection{Preliminary observations}

Let $n,m \geq 2$. Recall from Subsection~\ref{sub_phil} that the presentation $\langle X, Y \ \vert \ X^m = Y^n \rangle$ defines a Garside structure on $G(n,m)$. Denote $\mathcal{M}=\mathcal{M}(n,m)$ the corresponding Garside monoid with $\Delta= X^m = Y^n$. Let $u\in \mathcal{M}$ be an element such that $u\neq 1$ and $u$ does not admit $\Delta= X^m = Y^n$ as a left- (equivalently right-) divisor and denote by $\mathcal{M}_{\Delta\not\leq}$ the set of such elements of $\mathcal{M}$. Then $u$ can be uniquely written in the form $$Y^{n_1} X^{m_1} Y^{n_2} X^{m_2} \cdots Y^{n_k} X^{m_k}$$ where $k\geq 1$, $1 \leq m_i \leq m-1$ for all $i=1, \dots, k-1$, $0 \leq m_k \leq m-1$, $1 \leq n_i \leq n - 1$ for all $i=2, \dots, k$, $0 \leq n_1 \leq n-1$. We denote by $L(u)$ the first letter of the above word, and by $R(u)$ its last letter. Every factor $X^{m_i}$, $Y^{n_i}$ is a simple element (in fact, their concatenation as above yields both the left- and right-Garside normal form of $u$); we call such factors the \textit{Garside factors} of $u$. Let $\ell(u)$ denote the \textit{Garside length} of $u$, that is, $$\ell(u)= 2(k-1) + {\bf{1}}_{n_1 \neq 0} + {\bf{1}}_{m_k \neq 0}.$$
We denote by $\ell_w(u)$ its \textit{weighted length}, that is, the integer $Mn+Nm$, where $M=\sum m_i$ and $N= \sum n_i$. For later use, for an arbitrary $x\in \mathcal{M}$, one can define $\ell_w(x)$ by choosing a word for $x$, counting the number $M$ (respectively $N$) of occurrences of $X$ (resp. of $Y$), and setting $\ell_w(x) = Mn+ Nm$; the relation $X^m = Y^n$ guarantees that it is well-defined. Given a Laurent polynomial $P\in\mathbb{C}[t^{\pm 1}, q^{\pm 1}]$, denote by $d_q(P)$ its degree in $q$.

To show that the representation $X \mapsto M_X$, $Y \mapsto M_Y$ from Proposition~\ref{fund} is faithful, we first establish the following:

\begin{prop}\label{prop_read}
	Assume that $M_X, M_Y$ satisfy the assumptions of Proposition~\ref{fund}. In particular they define a representation $\rho : G(n,m)\longrightarrow \mathrm{GL}_2(\mathbb{C}[t^{\pm 1}, q^{\pm 1}])$. Let $u\in \mathcal{M}_{\Delta\not\leq}$ and let $M_u:=\rho(u)= \begin{pmatrix} a_{11} & a_{12} \\ a_{21} & a_{22} \end{pmatrix}$. Let $n_q(u)= (k-1) + \mathbf{1}_{n_1 \neq 0}$. Then the following holds: 
	\begin{enumerate}
		\item If $n_1 \neq 0$ and $m_k=0$, then $d_q(a_{21})= n_q(u)$ and $d_q(a_{ij}) < n_q(u)$ for all $(i,j) \neq (2,1)$. 
		\item If $n_1 \neq 0$ and $m_k \neq 0$, then $d_q(a_{22})= n_q(u)$, $d_q(a_{21}) \leq n_q(u)$ and $d_q(a_{ij}) < n_q(u)$ for all $(i,j) \neq (2,2), (2,1)$.
		\item If $n_1= 0$ and $m_k= 0$, then $d_q(a_{11})= n_q(u)$, $d_q(a_{21}) \leq n_q(u)$, and $d_q(a_{12}), d_q(a_{22}) < n_q(u)$.
		\item If $n_1= 0$ and $m_k\neq 0$, then $d_q(a_{12})= n_q(u)$ and $d_q(a_{ij}) \leq n_q(u)$ for all $(i,j) \neq (1,2)$.
	\end{enumerate}  
In particular, the highest degree in $q$ of the coefficients of $M_u$ is equal to $n_q(u)$. 
\end{prop}

We will symbolize the four situations occuring in the above proposition as follows:

\begin{center}
	\begin{figure}
\begin{tabular}{|c|c|c|c|}
	\hline
	Situation & $M_u$ & $L(u)$ & $R(u)$ \\ 
	\hline
1 & $\begin{pmatrix} < & < \\ = & < \end{pmatrix}$ & Y & Y \\
\hline
2 & $\begin{pmatrix} < & < \\ \leq & = \end{pmatrix}$ & Y & X \\
\hline
3 & $\begin{pmatrix} = & < \\ \leq & < \end{pmatrix}$ & X & Y \\
\hline
4 & $\begin{pmatrix} \leq & = \\ \leq & \leq \end{pmatrix}$ & X & X \\
\hline
\end{tabular}
\caption{Reading the first and last letters of the word from the matrix}
\label{tabular}
\end{figure}
\end{center}

\begin{cor}\label{cor_read}
	The elements $L(u)$ and $R(u)$ can be read off from $M_u$. The integer $\ell(u)$ can be read off from $M_u$. 
\end{cor}

\begin{proof}
	It follows from the above proposition that $Y=L(u)$ if and only if the maximal power of $q$
	among the coefficients of $M_u$ appears only in the second line of $M_u$. We can thus recover ${\bf{1}}_{n_1 \neq 0}$ from $M_u$. Similarly, we have $X=R(u)$ is and only if the maximal power of $q$ appears in the second column of $M_u$. We can thus also recover ${\bf{1}}_{m_k \neq 0}$.
	
	Since the maximal power in $q$ of a coefficient of $M_u$ is equal to $n_q(u)= (k-1) + {\bf{1}}_{n_1 \neq 0}$, we can thus recover $k$, hence also $\ell(u)$. 
\end{proof}

\begin{proof}[Proof of Proposition~\ref{prop_read}]
The proof will be by induction on $\ell(u)$. If $\ell(u)=1$, then $u=X^k$ for some $1\leq k \leq m-1$ or $u=Y^k$ for some $1\leq k \leq n-1$, hence we either lie in Situation $1$ or in Situation $4$, and the result holds by assumption on the matrices $M_X^k$, $M_Y^k$.

Hence assume that $\ell(u) \geq 1$ and let $\alpha$ be a Garside factor such that, denoting $v= \alpha u$, we have $\ell(v)= \ell(u)+1$. Assume that the result holds by induction on $u$. Let $M_u=\begin{pmatrix} a_{11} & a_{12} \\ a_{21} & a_{22} \end{pmatrix}$. If $\alpha=X^k$ for some $1 \leq k \leq m-1$, then  
$$M_v=M_X^k M_u = t^{nk} \begin{pmatrix} c_1(k) a_{11} + \underbrace{c_2(k)}_{\neq 0} a_{21} & c_1(k) a_{12} + \underbrace{c_2(k)}_{\neq 0} a_{22} \\ c_3(k) a_{11} + c_4(k) a_{21} & c_3(k) a_{12} + c_4(k) a_{22}  \end{pmatrix},$$ where the $c_i(k)$'s are complex numbers.

If $\alpha=Y^k$ for some $1\leq k \leq n-1$, then $$M_v=M_Y^k M_u = t^{km} \begin{pmatrix} \underbrace{P_1(k)}_{\mathrm{deg}\leq 0} a_{11} + \underbrace{P_2(k)}_{\mathrm{deg}\leq 0} a_{21} & \underbrace{P_1(k)}_{\mathrm{deg}\leq 0} a_{12} + \underbrace{P_2(k)}_{\mathrm{deg}\leq 0} a_{22} \\ \underbrace{P_3(k)}_{ \mathrm{deg} = 1} a_{11} + \underbrace{P_4(k)}_{\mathrm{deg}\leq 0} a_{21} & \underbrace{P_3(k)}_{ \mathrm{deg} = 1} a_{12} + \underbrace{P_4(k)}_{\mathrm{deg}\leq 0} a_{22}   \end{pmatrix},$$ where the $P_i(k)$'s are Laurent polynomials in $q$. 

Assume that $u$ lies in Situation 1. This means that the first Garside factor of $u$ is a nontrivial power of $Y$, hence that $\alpha=X^k$ for some $1\leq k \leq m-1$. We have $n_q(\alpha u)= n_q(u)$. By induction, the maximal degree of a coefficient of $M_u$ is $n_q(u)$ and is achieved in $a_{21}$ and only in this coefficient. The above calculation for $M_v=M_X^k M_u=(b_{ij})_{1\leq i,j \leq 2}$ then yields that the maximal degree of a coefficient in $M_v$ is still $n_q(u)$ and achieved in $b_{11}$ and possibly in $b_{21}$ as well, but not in the other coefficients. Since $v$ lands in Situation $3$, this shows the statement in this case. 

Assume that $u$ lies in Situation 2. The first Garside factor of $u$ is as in the previous case, we can thus argue similarly. We have $n_q(\alpha u)= n_q(u)$ and by induction, the maximal degree of a coefficient of $M_u$ is equal to $n_q(u)$ and is achieved in $a_{22}$ and possibly in $a_{21}$, but not in the other coefficients. The above calculation for $M_v=M_X^k M_u=(b_{ij})_{1\leq i,j \leq 2}$ then yields that the maximal degree of a coefficient in $M_v$ is still $n_q(u)$ and achieved in $a_{21}$ and possibly in any other coefficient of $M_v$ as well. Since $v$ lands in Situation $4$, this shows the statement in this case. 

Assume that $u$ lies in Situation 3. This means that the first Garside factor of $u$ is a nontrivial power of $X$, hence that $\alpha=Y^k$ for some $1\leq k \leq n-1$. We have $n_q(\alpha u)= n_q(u)+1$. By induction, the maximal degree of a coefficient of $M_u$ is $n_q(u)$ and is achieved in $a_{11}$ and possibly in $a_{21}$, but not in the other coefficients. The above formula for $M_v=M_Y^k M_u=(b_{ij})_{1\leq i,j \leq 2}$ then yields that the maximal degree of a coefficient in $M_v$ is $n_q(u)+1$ and is achieved in $b_{21}$ and only in $b_{21}$. Since $v$ lands in Situation $1$, this shows the statement in this case.

Finally, assume that $u$ lies in Situation 4. The first Garside factor of $u$ is as in the previous case, we can thus argue similarly. We have $n_q(\alpha u) = n_q(u)+1$.  By induction, the maximal degree of a coefficient of $M_u$ is $n_q(u)$ and is achieved in $a_{12}$ and possibly in any other coefficient. The above formula for $M_v=M_Y^k M_u=(b_{ij})_{1\leq i,j \leq 2}$ then yields that the maximal degree of a coefficient in $M_v$ is $n_q(u)+1$ and is achieved in $b_{22}$ and possibly in $b_{21}$ as well, but not in the other coefficients. Since $v$ lands in Situation $2$, this shows the statement in this case.     
\end{proof}

We now have all the tools at our disposal to show the faithfulness.

\subsection{Proof of Proposition~\ref{fund}}

\begin{prop}
Assume that $M_X, M_Y$ satisfy the assumptions of Proposition~\ref{fund}. Then 
\begin{enumerate}
	\item The representation $X \mapsto M_X$, $Y \mapsto M_Y$ is faithful, 
	\item The specialization of the representation at $t=q$ stays faithful. 
\end{enumerate}
\end{prop}

\begin{proof}
It suffices to show the second statement to get the first one. Since $\mathcal{M}$ is a Garside monoid, it suffices to show that the restriction of the representation to $\mathcal{M}$ is faithful (see Subsection~\ref{sub_phil}). Let $u,u'\in \mathcal{M}$ and assume that $M_u=M_{u'}$. We want to show that $u=u'$. We can assume that $u$ and $u'$ have no common non-trivial left-divisor. If $\Delta$ left-divides either $u$ or $u'$, say $u$, then $u'=1$, otherwise $u$ and $u'$ would have a common left-divisor. Now the condition $M_X^m = M_Y^n$ guarantees that the determinant of $t^{-m} M_Y$ is a complex number. It follows that in the specialized representation, the determinant of the matrix $M_u$ of any element $u$ of $\mathcal{M}$ is of the form $q^{\ell_w(u)} z$ where $z\in \mathbb{C}^*$ and $\ell_w(u)$ is the weighted length of $u$. If $\Delta$ left-divides $u$, we have $\ell_w(u) > 0$, hence we cannot have $\det(M_u)= 1$, which contradicts $M_u=M_{u'}$.  

We can thus assume that both $u, u'$ lie in $\mathcal{M}_{\Delta\not\leq}$. By a determinant argument again, we can assume that $u, u'\neq 1$. By Corollary~\ref{cor_read}, we can read the $L(u), L(u')$ from $M_u, M_{u'}$ (note that the specialization at $t=q$ just multiplies every coefficient by the same power of $q$, hence the content of the table in Figure~\ref{tabular} summarizing Proposition~\ref{prop_read} stays valid), hence $L(u)=L(u')$, thus $u$ and $u'$ have a common non-trivial left-divisor, a contradiction. This concludes the proof. 
\end{proof}

\subsection{Recovering the Garside normal form from the representation}

Assume again that $M_X$ and $M_Y$ satisfy the assumptions of Proposition~\ref{fund}. The aim of this subsection is to show that we can recover the Garside normal form of any element of $\mathcal{M}$ from the matrix $M_u$. 

Consider an element $u\in \mathcal{M}$. We can write it in the form $u' \Delta^k$\footnote{In fact, it can be written \textit{uniquely} in this form. Nevertheless, uniqueness is not needed, and the content of this subsection reproves it (provided that we can indeed build a representation satisfying the assumptions of Proposition~\ref{fund}): we see that we can recover the form $u' \Delta^k$ from $M_u$, which shows uniqueness. It also shows that $\mathcal{M}(n,m)$ embeds into $G(n,m)$.}, with $u'\in \mathcal{M}_{\Delta\not\leq}\cup\{1\}$. In our representation we then have $$M_u=M_{u'} M_\Delta^k.$$
Now since $\Delta=X^m=Y^n$ is central in $G(n,m)$, the matrix $M_\Delta=t^{nm}\begin{pmatrix}a & b \\ c & d \end{pmatrix}$ where $a,b,c,d\in\mathbb{C}$ has to commute with both $M(X)$ and $M(Y)$. Writing $P_i=P_i(1)$ and $c_i=c_i(1)$ for simplicity, this implies that \begin{align*}\begin{pmatrix} a c_1 + b c_3 & a c_2 + b c_4 \\ c c_1 + d c_3 & c c_2 + d c_4 \end{pmatrix} &=\begin{pmatrix}  a c_1 + c c_2 & b c_1 + d c_2 \\ a c_3 + c c_4 & b c_3 + d c_4  \end{pmatrix},\\ \begin{pmatrix} a P_1 + b P_3 & a P_2 + b P_4 \\ c P_1 + d P_3 & c P_2 + d P_4 \end{pmatrix} &=\begin{pmatrix}  a P_1 + c P_2 & b P_1 + d P_2 \\ a P_3 + c P_4 & b P_3 + d P_4  \end{pmatrix},\end{align*}
 yielding $$\begin{pmatrix} b c_3 & a c_2 + b c_4 \\ c c_1 + d c_3 & c c_2 \end{pmatrix} =\begin{pmatrix}  c c_2 & b c_1 + d c_2 \\ a c_3 + c c_4 & b c_3  \end{pmatrix}$$ and $$\begin{pmatrix}  b P_3 & a P_2 + b P_4 \\ c P_1 + d P_3 & c P_2 \end{pmatrix} =\begin{pmatrix}  c P_2 & b P_1 + d P_2 \\ a P_3 + c P_4 & b P_3 \end{pmatrix}.$$
Since $\deg(P_3)=1$ while $\deg(P_2) \leq 0$, this last equality of matrices forces to have $b=0$, yielding $c P_2=0$ and $a=d$. The first equality becomes $$\begin{pmatrix} 0 & a c_2 + b c_4 \\ c c_1 + a c_3 & c c_2 \end{pmatrix} =\begin{pmatrix}  c c_2 & a c_2 \\ a c_3 + c c_4 & 0  \end{pmatrix},$$ yielding $c=0$ since $c_2 \neq 0$. The matrix $M_\Delta$ is thus the scalar matrix $t^{nm}\begin{pmatrix} a & 0 \\ 0 & a \end{pmatrix}$.

It follows that the matrix $M_u$ has the same degrees of coefficients (in $q$) as $M_{u'}$. One can thus with Proposition~\ref{prop_read} recover $L(u')$ from $M_u$, and inductively recover the Garside normal form of $u'$ (recall that the Garside length is also recovered): we consider the matrix $M_{L(u')}^{-1} M_{u}$, and iterate the procedure until reaching a scalar matrix. This matrix is then $M_\Delta^k$, and one recovers the exponent by looking at the power of $t$ (which has to be $t^{knm}$). This allows one to recover the (right) Garside normal form $x_1 x_2 \cdots x_\ell \Delta^k$ of $u$. The $x_i$'s are just the maximal powers of $X$ or $Y$ that are obtained during the procedure. 

Note that the algorithm still works in the representation specizliaed at $q=t$: in this case the matrix $M_u$ does not have the same degrees of coefficients in $q$ as $M_{u'}$ if $k>0$, but all coefficients are just mutiplied by the same power of $q$ (which is $q^{knm}$), hence one can still apply Proposition~\ref{prop_read} until reaching a scalar matrix.

\section{Construction of representations satisfying Proposition~\ref{fund}}\label{const_rep}

The aim of this section is to construct representations satisfying the assumptions of Proposition~\ref{fund}. In subsection~\ref{sub_odd_dih} we show that the Burau representation of odd dihedral Artin groups can be treated in this way. In subsection~\ref{sub_faith_gen} we construct a simple faithful representation for all groups of the form $G(n,m)$. In subsection~\ref{sub_burau_torus} we construct a faithful representation generalizing the Burau representation of odd dihedral Artin groups to a large family of groups of the form $G(n,m)$, which includes some complex braid groups. In subsection~\ref{sub_unitarizable} we show that the representation constructed in the previous subsection is unitarizable. 

\subsection{Odd dihedral Burau representation}\label{sub_odd_dih}

Let $m=2$ and $n=2\ell+1$ be odd. Then $G(n,m)$ is isomorphic to the Artin group of dihedral type $I_2(n)$; its classical Artin presentation is Presentation~\ref{pres_3} in this case, which we shall write $$\langle \sigma_1, \sigma_2 \ \vert \ \underbrace{\sigma_1 \sigma_2 \cdots}_{n~\text{factors}} = \underbrace{\sigma_1 \sigma_2 \cdots}_{n~\text{factors}} \rangle.$$  The generator $Y$ corresponds to $\sigma_1 \sigma_2$ while $X$ corresponds to $\underbrace{\sigma_1 \sigma_2 \cdots}_{n~\text{factors}}$. In the classical (reduced) Burau representation $\rho_1 : G(n,m) \longrightarrow \mathrm{GL}_2(\mathbb{R}[q^{\pm 1}])$, one has $$\rho_1(Y)= q^2 \begin{pmatrix} -1 & -q^{-1} \\ c q & c-1\end{pmatrix}, \ \rho_1(X) = q^n \begin{pmatrix} 0 & - \frac{\zeta_n^\ell - \zeta_n^{-\ell}}{\zeta_n - \zeta_n^{-1}} \\ c \frac{\zeta_n^{\ell+1} - \zeta_n^{-\ell-1}}{\zeta_n - \zeta_n^{-1}} & 0  \end{pmatrix},$$ where $c^2 = 4 \cos(\pi / n)$ (see for instance~\cite[(3.2), (3.3.1) and and Proposition 4.1]{LX}). To apply Lemma~\ref{fund} with the matrices $M_X:=\rho_1(X)$ and $M_Y:=\rho_1(Y)$, one needs to calculate the powers of $M_Y$, which is done in~\cite[Lemma 3.3]{LX}: for $k \geq 0$, setting $[\zeta_n]_k:=\frac{\zeta_n^k - \zeta_n^{-k}}{\zeta_n - \zeta_n^{-1}}$ we have $$M_Y^k = q^{2k} \begin{pmatrix} -[\zeta_n]_k - [\zeta_n]_{k-1} & -q^{-1} [\zeta_n]_k \\ cq [\zeta_n]_k & [\zeta_n]_k + [\zeta_n]_{k+1} \end{pmatrix}.$$
Since $m=2$, the only power of $M_X$ to consider is $M_X$ itself, while for $M_Y$ it is $M_Y^k$ for all $1\leq k \leq n-1$. Since $n$ is odd, we have $c\neq 0$, $\zeta_n^\ell \neq \zeta_n^{-\ell}$ and $[\zeta_n]_k \neq 0$ for all $1\leq k \leq n-1$, hence the conditions of Lemma~\ref{fund} are fulfilled. This yields a new proof that the Burau representation of odd dihedral Artin groups is faithful. This will be reobtained in Subsection~\ref{sub_burau_torus} below (see Example~\ref{ex_burau_part}) by viewing the odd dihedral Burau representation as a particular case of a representation built for a larger family of groups inside the family of groups of the form $G(n,m)$.   

\subsection{A faithful representation for arbitrary groups of the form $X^m=Y^n$}\label{sub_faith_gen}

Let $n,m\geq 2$ be two integers. We shall construct a faithful representation $\rho_2 : G(n,m) \longrightarrow \mathrm{GL}_2(\mathbb{C}[t^{\pm 1}, q^{\pm 1}])$. Note that if $\ell \geq 2$ is an integer, for $\zeta_\ell= e^{2i\pi/\ell}$, for all $k\geq 0$ we have $$\begin{pmatrix} \zeta_\ell & 1 \\ 0 & \zeta_{\ell}^{-1} \end{pmatrix}^k =\begin{pmatrix} \zeta_\ell^k & [\zeta_\ell]_k \\ 0 & \zeta_\ell^{-k} \end{pmatrix}, \  \begin{pmatrix} \zeta_\ell & 0 \\ 1 & \zeta_{\ell}^{-1} \end{pmatrix}^k =\begin{pmatrix} \zeta_\ell^k & 0 \\ [\zeta_\ell]_k & \zeta_\ell^{-k} \end{pmatrix}.$$ 
Let $\rho_2(X) := t^n \begin{pmatrix}  \zeta_{2m} & 1 \\ 0 & \zeta_{2m}^{-1}  \end{pmatrix}$. For all $k \geq 0$ we thus have $$ \rho_2(X)^k = t^{nk} \begin{pmatrix}\zeta_{2m}^k & [\zeta_{2m}]_k \\ 0 & \zeta_{2m}^{-k} \end{pmatrix}.$$ Let $\rho_2(Y):=t^m \begin{pmatrix} \zeta_{2n} & 0 \\ q & \zeta_{2n}^{-1} \end{pmatrix}$. Noting that $\rho_2(Y) = P (\rho_2(Y))_{q=1} P^{-1}$, where $P=\begin{pmatrix} 1 & 0 \\ 0 & q \end{pmatrix}$, we get that for all $k\geq 0$, $$\rho_2(Y)^k = t^{mk} \begin{pmatrix} \zeta_{2n}^k & 0 \\ q[\zeta_{2n}]_k & \zeta_{2n}^{-k} \end{pmatrix}.$$
Setting $M_X=\rho_2(X)$ and $M_Y=\rho_2(Y)$ we then have \begin{itemize}
	\item $M_Y^n = - t^{nm} I_2=M_X^m,$
	\item For all $1 \leq k \leq m-1$ (resp. for all $1\leq k \leq n-1$), the matrix $M_X^k$ (resp. $M_Y^k$) satisfies assumption $2$ (resp. $3$) of Proposition~\ref{fund} (observe that $[\zeta_{2\ell}]_k=0$ if and only if $\ell$ divides $k$). 
\end{itemize}
By Proposition~\ref{fund}, we thus get that $\rho_2$ is faithful. 
\begin{rmq}[Specializing $t$ or $q$ to $1$]\label{rmq_spec}
Specializing $t$ to $1$ yields a representation of $G(n,m)$ which is not faithful since $(\rho_2(Y))_{t=1}^n = - I_2 = (\rho_2(X))_{t=1}^m$, hence $(\rho_2(Y))_{t=1}$ and $(\rho_2(X))_{t=1}$ have finite order while  $X$ and $Y$ have infinite order inside $G(n,m)$. The question of whether the above representation with $q$ specialized at $1$ is faithful or not seems interesting; if $n=m=3$, then the representation is not faithful as setting $A:=(\rho_2(X))_{q=1}$ and $B:=(\rho_2(Y))_{q=1}$, one checks that $ABAB=BABA$, while $XYXY\neq YXYX$ inside $G(3,3)$ (to see this one can use the fact that $\langle X, Y \ \vert \ X^3 = Y^3 \rangle$ is a Garside presentation of $G(3,3)$, hence that the monoid with the same presentation embeds into $G(3,3)$; both elements $XYXY$ and $YXYX$ lie in this submonoid, and differ inside it). For $n=2$ and $m=3$, we do not know whether it is faithful or not. It seems reasonable to expect the representation not to be faithful if $n$ and $m$ are not coprime, but for coprime $n$ and $m$ we do not have a clear idea of what to expect, leading to the following question.   
\end{rmq}

\begin{question}
Assume that $n$ and $m$ are coprime. Does the representation $X \mapsto (\rho_2(X))_{q=1}$, $Y \mapsto (\rho_2(Y))_{q=1}$ stay faithful?
\end{question}

\subsection{A Burau representation for some torus knot groups}\label{sub_burau_torus}

In this section we shall assume that $n, m\geq 2$ are coprime and that $m$ is not divisible by $3$.  

Let $a, b\in\mathbb{Z}$ be such that $an-bm=1$. Consider the matrices 
$$M_X= (-1)^b i^{n+1} t^n \begin{pmatrix} U' & V \\ V & -U \end{pmatrix}, \ M_Y= (-1)^a i^m t^m \begin{pmatrix} A_1 [\boldsymbol{\mu}^b]_1 + \mu_1 & q^{-1} V [\boldsymbol{\lambda}]_a  [\boldsymbol{\mu}^b]_1  \\ - \frac{A_1 A_2}{ q^{-1} V [\boldsymbol{\lambda}]_a} [\boldsymbol{\mu}^b]_1 & \mu_1 - A_2 [\boldsymbol{\mu}^b]_1 \end{pmatrix},$$ where 
\begin{itemize}
	\item $U= \zeta_{4m}^{3m-2} + \zeta_{4m}^{3m+2} = \zeta_{4m}^{3m-2} - (\zeta_{4m}^{3m-2})^{-1}$,
\item $U' = U + \lambda_1 + \lambda_2$, where $\lambda_1= \zeta_{4m}^{3m-6}$ and $\lambda_2=  \zeta_{4m}^{3m+6}$. Note that $\lambda_2 = - \lambda_1^{-1}$. 
\item $V= -1 - \zeta_{m} - \zeta_{m}^{-1}$ if $m > 2$, and $-1$ otherwise. Since $m\neq 3$ we have $V \neq 0$. Note that $U' U + V^2 = 1$. 
	\item $\mu_1= \zeta_n$ and $\mu_2=\zeta_n^{-1}$.
\item For $k\in\mathbb{Z}$, $[\boldsymbol{\lambda}]_k=\frac{\lambda_2^k - \lambda_1^k}{\lambda_2 - \lambda_1}$, and $[\boldsymbol{\mu}^b]_k= \frac{\mu_2^k - \mu_1^k}{\mu_2^b - \mu_1^b}$. Note that $[\boldsymbol{\lambda}]_0=0$, $[\boldsymbol{\lambda}]_1=1=[\boldsymbol{\lambda}]_{-1}$, $[\boldsymbol{\lambda}]_k=0$ if and only if $m$ divides $3k$, if and only if $m$ divides $k$ since we assumed that $3$ does not divide $m$. In particular $[\boldsymbol{\lambda}]_a\neq 0$ since $(a,m)=1$. Similarly we have $[\boldsymbol{\mu}^b]_k=0$ if and only if $n$ divides $2k$.
\item $A_i= q^{-1} U' [\boldsymbol{\lambda}]_a + q^{-1} [\boldsymbol{\lambda}]_{a-1} - \mu_i^b$.  
\end{itemize}
\begin{prop}\label{prop_form_powers}
	Let $k\in \mathbb{Z}$. We have \begin{align*} 
		M_X^k = (-1)^{bk} i^{nk+k} t^{nk}\begin{pmatrix} U' [\boldsymbol{\lambda}]_k + [\boldsymbol{\lambda}]_{k-1} &  V [\boldsymbol{\lambda}]_k \\ V [\boldsymbol{\lambda}]_k & -U' [\boldsymbol{\lambda}]_k + [\boldsymbol{\lambda}]_{k+1}\end{pmatrix},\\ 
		M_Y^k= (-1)^{ak} i^{mk} t^{mk} \begin{pmatrix} A_1 [\boldsymbol{\mu}^b]_{k} + \mu_1^{k} & q^{-1} V [\boldsymbol{\lambda}]_a  [\boldsymbol{\mu}^b]_{k}  \\ - \frac{A_1 A_2}{ q^{-1} V [\boldsymbol{\lambda}]_a} [\boldsymbol{\mu}^b]_{k} & \mu_1^{k} - A_2 [\boldsymbol{\mu}^b]_{k} \end{pmatrix}.
	\end{align*}
\end{prop}
\begin{proof}
	Straightforward computation by induction on $k$.  
\end{proof}

\begin{prop}\label{prop_burau_torus}
	The above matrices define a representation $\rho_3$ of the torus knot group $G(n,m)=\langle X, Y \ \vert \ X^m = Y^n \rangle$. 
\end{prop}

\begin{proof}
	We have $[\boldsymbol\lambda]_m=0$, $[\boldsymbol\lambda]_{m-1}=(-1)^{m+1} i^m = [\boldsymbol\lambda]_{m+1}$ and $[\boldsymbol\mu^b]_n=0$, thus using Proposition~\ref{prop_form_powers} we get 
	\begin{align*} M_X^m &= (-1)^{bm} i^{nm+m} t^{nm} \begin{pmatrix} (-1)^{m+1} i^m & 0 \\ 0 & (-1)^{m+1} i^m \end{pmatrix} = (-1)^{bm+1} i^{nm+m} (-i)^m t^{nm} I_2 \\ &= (-1)^{an} i^{nm} t^{nm} I_2 = M_Y^n.
	\end{align*}  
\end{proof}

We now apply Proposition~\ref{fund} to the above-constructed representation, under the further assumption that $n$ is odd: 

\begin{prop}
Let $n$, $m\geq 2$ be coprime and assume that $n$ is odd and that $m$ is not divisible by $3$. The above-constructed representation of $G(n,m)$ is faithful. 
\end{prop}

\begin{proof}
We check the conditions of Proposition~\ref{fund}. The first condition was obtained in Proposition~\ref{prop_burau_torus}. The second condition is clear from Proposition~\ref{prop_form_powers}, since all the entries of $t^{-nk} M_X^k$ are complex numbers, $V\neq 0$, and $[\boldsymbol\lambda]_k =0$ if and only if $m$ divides $3k$, if and only if $m$ divides $k$ since $m$ is not divisible by $3$. The last condition is also obtained from Proposition~\ref{prop_form_powers} since $[\boldsymbol{\lambda}]_a\neq 0$, $[\boldsymbol \mu^b]_k=0$ if and only if $n$ divides $2k$, if and only if $n$ divides $k$ since $n$ is odd. 
\end{proof}

\begin{rmq}
\begin{enumerate}
	\item Since $G(n,m)\cong G(m,n)$, in some cases we get two faithful two-dimensional representations of $G(n,m)$. We did not investigate whether these two representations are isomorphic or not in general. Moreover, using this isomorphism, in some cases where the pair $(n,m)$ does not satisfy the assumptions that $n$ is odd and $m$ is not divisible by $3$, we can permute them. Using this we can define at least one "Burau-like" representation when none of the parameters $n$ and $m$ is divisible by $6$.
	\item The above-defined representation is also defined without the assumption that $n$ is odd, but is not faithful in this case: if $n=2n'$, we have $[\boldsymbol \mu^b]_{n'}=0$, hence $M_Y^{n'}$ is a scalar matrix, hence it commutes with $M_X$, while in $G(n,m)$ we have $X Y^{n'} \neq Y^{n'} X$ (arguing for instance as in Remark~\ref{rmq_spec}).   
\end{enumerate}
\end{rmq}

For $k\in \mathbb{Z}$ we have $[\boldsymbol{\lambda}]_{-k}=(-1)^{k+1} [\boldsymbol{\lambda}]_k$. Using Proposition~\ref{prop_form_powers}, we thus have \begin{align*} M_X^{-a}&= (-1)^{ab} i^{-an-a} t^{-an} \begin{pmatrix} U' (-1)^{a+1} [\boldsymbol{\lambda}]_a +(-1)^a [\boldsymbol{\lambda}]_{a+1} &  V (-1)^{a+1} [\boldsymbol{\lambda}]_a \\ V (-1)^{a+1} [\boldsymbol{\lambda}]_a & -U' (-1)^{a+1} [\boldsymbol{\lambda}]_a +(-1)^a [\boldsymbol{\lambda}]_{a-1}\end{pmatrix},\end{align*} and setting $A=q^{-1} U' [\boldsymbol{\lambda}]_a + q^{-1} [\boldsymbol{\lambda}]_{a-1}$ we have \begin{align*} 
	M_Y^b &=(-1)^{ab} i^{bm} t^{bm} \begin{pmatrix}   A & q^{-1} V [\boldsymbol{\lambda}]_a \\ - q \frac{( A - \zeta_n^b)(A - \zeta_n^{-b})}{V [\boldsymbol\lambda]_a} & \zeta_n^b + \zeta_n^{-b} - A \end{pmatrix}. 
\end{align*}
Using that $[\boldsymbol\lambda]_a^2 - [\boldsymbol\lambda]_{a-1} [\boldsymbol\lambda]_{a+1}=(-1)^{a+1}$ and $[\boldsymbol\lambda]_{a} (\lambda_1 +\lambda_2) +  [\boldsymbol\lambda]_{a-1} - [\boldsymbol\lambda]_{a+1}=0$ we calculate 
\begin{align} \label{meridien} M_Y^b M_X^{-a} = i^{-a-1} t^{-1}\begin{pmatrix} q^{-1} & 0 \\ \star & (-1)^a q \end{pmatrix} \end{align}
We thus have:
\begin{lemma}\label{lem:mer}	
	\begin{enumerate}
		\item If $a$ is odd, then $(M_Y^b M_X^{-a})_{t=q=1}$ is the matrix of a complex reflection of order $2$. 
		\item If $a$ is even, then $(M_Y^b M_X^{-a})_{t=1, q=i}$ is the matrix of a complex reflection of order $2$. 
	\end{enumerate}
\end{lemma}

Recall the \textit{$2$-toric reflection group} from Definition~\ref{2toric}.

\begin{cor}
\begin{enumerate}	
\item If $a$ is odd, then $(M_X)_{t=q=1}$ and $(M_Y)_{t=q=1}$ define a representation of the $2$-toric reflection group $W(2,n,m)$.
\item If $a$ is even, then $(M_X)_{t=1, q=i}$ and $(M_Y)_{t=1, q=i}$ define a representation of the $2$-toric reflection group $W(2,n,m)$. 
\end{enumerate} 
\end{cor}

\begin{proof}
This follows immediately from Lemma~\ref{lem:mer} together with Remark~\ref{rmq_mer}. 
\end{proof}

\begin{rmq}\label{rmq_q_iq}
The above corollary suggests that, whenever $a$ is odd, the representation $\rho_3$ is a "Burau representation" for $G(n,m)$. Whenever $a$ is even, it suggests to replace $q$ by $iq$ to get a "Burau representation"; it is indeed expected from a generalization of the Burau representation to specialize to the generalization of the dihedral group obtained when $m=2$ and $n$ is odd, which is given by the attached $2$-toric reflection group. In the examples below we explore this property by showing that $\rho_3$ indeed generalizes the dihedral Burau representation, and that for the cases $(n,m)=(3,4)$ and $(n,m)=(3,5)$, the same phenomenon appears: in this case $G(n,m)$ is the complex braid group of the exceptional complex reflection groups $G_{12}$ and $G_{22}$, and there should exist a faithful $2$-dimensional representation deforming the reflection groups. We check that it is what $\rho_3$ does. 
\end{rmq}

We now study some particular cases of the representation that we just constructed. 

\begin{exple}\label{ex_burau_part}
Let $n \geq 3$ be odd and assume that $m=2$. In this case $G(n,2)$ is isomorphic to the Artin group of dihedral type $I_2(n)$. We have $U=U'=0$, $V=-1$. Denoting $n=2n' + 1$ we have $a=1, b=n'$. This yields the matrices \begin{align*} M_X = t^n \begin{pmatrix} 0 & 1 \\ 1 & 0 \end{pmatrix}, \ M_Y &= t^2 \begin{pmatrix} - \zeta_n^{n'} \frac{\zeta_n^{-1}- \zeta_n}{ \zeta_n^{-n'} - \zeta_n^{n'}} + \zeta_n & -q^{-1} \frac{\zeta_n^{-1}- \zeta_n}{\zeta_n^{-n'} - \zeta_n^{n'}} \\ q \frac{\zeta_n^{-1}- \zeta_n}{\zeta_n^{-n'} - \zeta_n^{n'}} & \zeta_n + \zeta_n^{-n'}\frac{\zeta_n^{-1}- \zeta_n}{\zeta_n^{-n'} - \zeta_n^{n'}} \end{pmatrix}\\ &= t^2 \begin{pmatrix}\frac{\zeta_n^{-(n'-1)} - \zeta_n^{n'-1}}{\zeta_n^{-n'} - \zeta_n^{n'}} & -q^{-1} \frac{\zeta_n^{-1}- \zeta_n}{\zeta_n^{-n'} - \zeta_n^{n'}} \\q \frac{\zeta_n^{-1}- \zeta_n}{\zeta_n^{-n'} - \zeta_n^{n'}} & \frac{\zeta_n^{-(n'+1)}- \zeta_n^{n'+1}}{\zeta_n^{-n'} - \zeta_n^{n'}} \end{pmatrix}.\end{align*}
Conjugating these two matrices by the matrix $P=\begin{pmatrix} 1 & 0 \\ 0 & \frac{\zeta_n^{-1}- \zeta_n}{\zeta_n^{n'} - \zeta_n^{-n'}} \end{pmatrix}$ yields after simplification the two matrices $$M_X'= t^n \begin{pmatrix} 0 & \frac{\zeta_n^{n'}- \zeta_n^{-n'}}{\zeta_n^{-1} - \zeta_n} \\ \frac{\zeta_n^{-1} - \zeta_n}{ \zeta_n^{n'} - \zeta_n^{-n'}} & 0 \end{pmatrix}, \ M_Y' = t^2 \begin{pmatrix} \zeta_n + 1 + \zeta_n^{-1} & q^{-1} \\ -q (\zeta_n + 2 + \zeta_n^{-1}) & -1  \end{pmatrix}.$$
 Recall that, in the notation of Subsection~\ref{sub_odd_dih}, the isomorphism between $G(n,2)$ and the dihedral Artin group of type $I_2(n)$ sends $Y$ to $\sigma_1 \sigma_2$ and $X$ to $\underbrace{\sigma_1 \sigma_2 \cdots}_{n`~\text{factors}}$. Replacing $q$ by $q^{-1}$ and specializing $t$ to $q$ yields the matrices for the lift of $w_0$ and $st$ of the reduced Burau representation of $I_2(n)$, which can also be found in~\cite[(3.2), (3.3.1) and Proposition 4.1]{LX}\footnote{Note that, in Subsection~\ref{sub_odd_dih}, the matrix $q^{-2} M_Y$ was corresponding to the one on the left in~\cite[3.2]{LX}, while here $q^{-2} (M_Y')_{q \leftrightarrow q^{-1}, t \mapsto q}$ it is the one on the right; it does not matter since $\sigma_1$ and $\sigma_2$ play symmetric roles in the Artin presentation.}. We thus get that, when the group is $G(n,2)$, the representation $\rho_3$ is isomorphic to the reduced Burau representation of the Artin group of type $I_2(n)$.   
\end{exple}

\begin{exple}[Burau representation for $G_{12}$]\label{g12}
The previous example shows that the representation $\rho_3$ is not new when $m=2$. The "smallest" case for which our representation is new is the case $n=3, m=4$. In this case the group $G(3,4)$ is isomorphic to the complex braid group of the exceptional complex reflection group $G_{12}$. We have $U=- \sqrt{-2}$, $U'=0$, $V=-1$, $a=b=-1$, $[\boldsymbol{\lambda}]_{-2}= U$, $[\boldsymbol{\lambda}]_{-1}=1$, $[\boldsymbol{\mu}^{-1}]_{1}=-1$. This yields the matrices $$M_X= t^3 \begin{pmatrix} 0 & 1 \\ 1 & U \end{pmatrix}, M_Y = t^4 \begin{pmatrix} 1+q^{-1} U  & -q^{-1} \\ - 2q^{-1} + q + U & -q^{-1}U\end{pmatrix}.$$ 
Specializing $q$ and $t$ to $1$ yields matrices generating the complex reflection group $G_{12}$. One way to check this is as follows: the group $G_{12}$ has order $48$, with a reflection presentation given by Presentation~\ref{pres_2} with $n=3$, $m=4$, and the additional relations $x_i^2=1$ for all $i$. The $x_i$'s are reflections and there is only one conjugacy class of reflections. Using Remark~\ref{rmq_mer}, we get that an alternative presentation of $G_{12}$ is given by $X^4 = Y^3$ and $(Y^{-1} X)^2=1$, with $Y^{-1} X$ sent to some $x_i$. Since $M_X$ and $M_Y$ satisfy the first relation (hence so do their specializations) and $M_Y^{-1} M_X = M_Y^b M_X^{-a}$ has the form~\ref{meridien} hence specializes to a matrix of order $2$, we get that $\langle (M_X)_{t=q=1}, (M_Y)_{t=q=1} \rangle$ is a quotient of $G_{12}$. But one checks using a computer that this group has order $48$, hence it is an isomorphism and the conjugacy class of the $x_i$'s is the same as the conjugacy class of $M_Y^{-1} M_X$ which are indeed complex reflections.    
\end{exple}

\begin{exple}[Burau representation for $G_{22}$]\label{g22}
Consider the case $n=3, m=5$. In this case $G(3,5)$ is isomorphic to the complex braid group of the exceptional complex reflection group $G_{22}$. We have $U= \zeta_{20}^{13}+\zeta_{20}^{17}$, $U'=-i$, $V=\zeta_5^2+\zeta_5^3$, $a=2, b=1$, $[\boldsymbol{\lambda}]_{2}= \zeta_{20}^9 + \zeta_{20}$, $[\boldsymbol{\lambda}]_1={[\boldsymbol{\mu}^1]}_1=1$. Since $a$ is even, following Remark~\ref{rmq_q_iq}, we replace $q$ by $iq$ in the defining matrices, yielding the matrices 
\begin{align*}(M_X)_{q \leftrightarrow q^{-1}} = -t^3 \begin{pmatrix} -i & \zeta_5^2 + \zeta_5^3 \\ \zeta_5^2 + \zeta_5^3 & -(\zeta_{20}^{13} + \zeta_{20}^{17})\end{pmatrix}, \\ (M_Y)_{q \leftrightarrow q^{-1}} = i t^5 \begin{pmatrix} - q^{-1} (\zeta_{20}^9 + \zeta_{20}) - i  q^{-1} & - iq^{-1} (\zeta_5^2+\zeta_5^3)(\zeta_{20}^9 + \zeta_{20}) \\ \frac{(- q^{-1} (\zeta_{20}^9 + \zeta_{20}) - i  q^{-1} - \zeta_3)(- q^{-1} (\zeta_{20}^9 + \zeta_{20}) - i  q^{-1} - \zeta_3^{-1})}{iq^{-1}(\zeta_5^2+\zeta_5^3)(\zeta_{20}^9 +  \zeta_{20})}
		&-1 + q^{-1} (\zeta_{20}^9 + \zeta_{20}) + i q^{-1} \end{pmatrix} \end{align*}  
As for Example~\ref{g12}, specializing $q$ and $t$ to $1$ yields matrices generating the complex reflection group $G_{22}$: this can be checked in the same way as for $G_{12}$. 
\end{exple}

\subsection{Unitarizability}\label{sub_unitarizable}

If $a$ is odd, then the representation constructed in Subsection~\ref{sub_burau_torus} is unitarizable: setting $J= \begin{pmatrix} 0 & 1 \\ 1 & 0 \end{pmatrix}$, it is straightforward to check that ${ }^t\!{\overline{M_X}} J M_X = J$ and ${ }^t\!{\overline{M_Y}} J M_Y$, where for $P\in \mathbb{C}[t^{\pm 1}, q^{\pm 1}]$, we set $\overline{P(q,t)}= P(-q, t^{-1})$. Note that in this case we have $\overline{A_i}= A_j$, where $\{i,j\}=\{1,2\}$.

If $a$ is even, then we again replace $q$ by $iq$ in the matrices from Subsection~\ref{sub_burau_torus}. Then the same matrix $J$ as above yields unitarizability and we again have that $\overline{A_i}= A_j$, $\{i,j\}=\{1,2\}$. 

\section{Hecke algebras of $2$-toric reflection groups}\label{hecke}

Let $n,m\geq 2$ with $n$ and $m$ coprime. Let $\mathcal{A}=\mathbb{Z}[q^{\pm 1}]$. 

\begin{definition}
The \defn{Hecke algebra} of the $2$-toric reflection group $W(2,n,m)$ is the associative, unital $\mathcal{A}$-algebra $\mathcal{H}(2,n,m)$ with generators $x_i$, $i=1, \dots, n$ and relations \begin{itemize}
	\item The defining relations of $G(n,m)$ from Presentation~\ref{pres_2}, 
	\item The relations $x_i^2 = (q^{-2} - 1) x_i + q^{-2}$, for all $i=1, \dots, n$. 
\end{itemize}
\end{definition}
In the cases where $W(2,n,m)$ is finite, then it is a finite (complex) reflection group $W$, and this is nothing but the usual (one-parameter) Hecke algebra from~\cite{BMR}, which coincides with the Iwahori-Hecke algebra in the real case. In this case it is a free $\mathcal{A}$-module of rank $|W|$, with a basis deforming the basis of $\mathbb{Z}[W]$ (obtained by specializing $q$ to $1$) formed by the elements of the group: see~\cite{Marin_report} and the references therein. It is tempting to conjecture that it still holds in the infinite case that $\mathcal{H}(2,n,m)$ is a free $\mathcal{A}$-module with a basis deforming the basis of $\mathbb{Z}[W(2,n,m)]$ formed by the elements of the group, but we do not know how to attack such a question since already in the finite case, the proof is case-by-case (see~\cite{Marin_report} for a detailed list of references). 

A Burau representation for complex braid groups should factor through the Hecke algebra. For $G_{12}$ and $G_{22}$ studied in the previous section, matrix models for two-dimensional representations of the Hecke algebra of these groups may be found in~\cite{MM}\footnote{I thank ivan Marin for pointing this out to me.}.  

As a corollary of the previous sections, we can show the following:  

\begin{prop}
Assume that $n,m \geq 2$, are coprime, that $n$ is odd, and that $m$ is not divisible by $3$. Then the group $G(n,m)$ embeds into $\mathcal{H}(2,n,m)^\times$ via $x_i \mapsto x_i$. In particular, the complex braid groups of $G_{12}$ and $G_{22}$ embed into their one-parameter Hecke algebra. 
\end{prop}
In the case where $m=2$ (and thus $n$ is odd), this is known by work of Lehrer and Xi~\cite{LX}: it is the embedding of the Artin group of dihedral type $I_2(n)$ inside the corresponding Iwahori-Hecke algebra. For the cases where $(n,m)=(3,4), (3,5)$, this reformulates into saying that the braid group of the exceptional complex reflection groups $G_{12}$ and $G_{22}$ embed into their (one-parameter) Hecke algebra. 

\begin{proof}
We first give a presentation of $\mathcal{H}(2,n,m)$ using the generators $X$ and $Y$ of $G(n,m)$. By Remark~\ref{rmq_mer}, the defining relations with this set of generators become $X^m = Y^n$, and $(Y^b X^{-a})^2 = (q^{-2} - 1)Y^b X^{-a} + q^{-2}$. We now construct a map from $\mathcal{H}(2,n,m)$ to $\mathrm{M}_2(\mathbb{C}[q^{\pm 1}])$ using the representation $\rho_3$ from Subsection~\ref{sub_burau_torus}.    

By~\eqref{meridien}, we have that $\rho_3(Y)^b \rho_3(X)^{-a}$ has the form $$i^{-a-1} t^{-1} \begin{pmatrix} q^{-1} & 0 \\ \star & (-1)^a q \end{pmatrix}.$$
We know that $X \mapsto \rho_3(X)$, $Y \mapsto \rho_3(Y)$ defines a faithful representation of $G(n,m)$ for generic values of $q$ and $t$. The idea is to build a map $\psi:\mathcal{H}(2,n,m) \longrightarrow \mathrm{M}_2(\mathbb{C}[q^{\pm 1}])$ such that, composing with the map $\varphi$ from $G(n,m)$ to $\mathcal{H}(2,n,m)$ sending every $x_i$ to $x_i$, we recover $\rho_3$. By faithfulness of $\rho_3$ this will yield the injectivity of $\varphi$. 

Since $\rho_3$ defines a representation of $G(n,m)$ for generic values of $q$ and $t$, it suffices to make choices such that the above matrix has the form $\begin{pmatrix} q^{-2} & 0 \\ \star & -1 \end{pmatrix}$. This can be achieved as follows: 
\begin{itemize}
\item If $a$ is odd, say $a = 2 \ell + 1$, then if $\ell$ is odd, setting $t=q$ yields a matrix of the above form. If $\ell$ is even, then replacing $t$ by $-t$, and then setting $t=q$ yields of matrix of the required form. 
\item If $a$ is even, say $a= 2 \ell$, then if $\ell$ is even, replacing $q$ by $iq$ and $t$ by $-t$, then setting $t=q$ yields a matrix of the above form. If $\ell$ is odd, replacing $q$ by $iq$ and then setting $t=q$ also yields a matrix of the required form.  
\end{itemize}
We need to justify that under the various specializations made, the modified representation of $G(n,m)$ into $\mathrm{GL}_2(\mathbb{C}[q^{\pm 1}])$ is still faithful. Let $M_X', M_Y'$ denote the matrices obtained from $M_X, M_Y$ after replacing $q$ by some element in $\{q, iq\}$ and $t$ by some element in $\{t, -t\}$ according to the above rules depending on the parities of $a$ and $\ell$. Since $M_X$ and $M_Y$ satisfy the relations $M_X^m = M_Y^n$ for generic values of $q$ and $t$, it is clear that we still have $(M_X')^m =(M_Y')^n$. Moreover, it is clear that the matrices $M_X'$ and $M_Y'$ still satisfy assumptions $(2)$ and $(3)$ of Proposition~\ref{fund}. Hence by the latter $M_X'$ and $M_Y'$ still define a faithful representation of $G(n,m)$ and their specialization to $t=q$ as well.   

We thus have that $X \mapsto (M_X')_{t=q}$, $Y \mapsto (M_Y')_{t=q}$ is a faithful representation of $G(n,m)$ factoring through $\mathcal{H}(2,n,m)$, whence the result. 
\end{proof}

\end{document}